\documentclass[12pt]{amsart}
\usepackage[active]{srcltx}
\usepackage{latexsym}
\usepackage{hyperref}
\usepackage{url}
\usepackage{amsmath}

\newcommand{\ignore}[1]{}

\newcommand{\hide}[1]{}

\DeclareMathOperator{\E}{E}

\DeclareMathOperator{\AH}{AH}

\newcommand{\F}{\mathbb F}

\newcommand{\Z}[0]{\mathbb Z}

\newcommand{\Q}{\mathbb{Q}}

\newtheorem{dummy}{Dummy}

\newtheorem{lemma}[dummy]{Lemma}

\newtheorem{prop}[dummy]{Proposition}
\newtheorem{cor}[dummy]{Corollary}

\theoremstyle{definition}

\newtheorem{conj}[dummy]{Conjecture}

\theoremstyle{remark}
\newtheorem{rem}[dummy]{Remark}
\newtheorem*{rem*}{Remark to ourselves}

\begin{document}

\bibliographystyle{amsalpha}

\author{Marina Avitabile}
\email{marina.avitabile@unimib.it}
\address{Dipartimento di Matematica e Applicazioni\\
  Universit\`a degli Studi di Milano - Bicocca\\
 via Cozzi 55\\
  I-20125 Milano\\
  Italy}
\author{Sandro Mattarei}
\email{smattarei@lincoln.ac.uk}
\address{Charlotte Scott Centre for Algebra\\
University of Lincoln \\
Brayford Pool
Lincoln, LN6 7TS\\
United Kingdom}

\title{On some coefficients of the Artin-Hasse series modulo a prime}


\begin{abstract}
Let $p$ be an odd prime, and let
$\sum_{n=0}^{\infty} a_{n}X^{n}\in\F_p[[X]]$
be the reduction modulo $p$ of the Artin-Hasse exponential.
We obtain a polynomial expression for $a_{kp}$ in terms of those $a_{rp}$ with $r<k$,
for even $k<p^2-1$.
A conjectural analogue covering the case of odd $k<p$ can be stated in various polynomial forms,
essentially in terms of the polynomial
$\gamma(X)
=\sum_{n=1}^{p-2}(B_{n}/n)X^{p-n}$,
where $B_n$ denotes the $n$-th Bernoulli number.

We prove that $\gamma(X)$ satisfies  the functional equation
$\gamma(X-1)-\gamma(X)=\pounds_1(X)+X^{p-1}-w_p-1$
in $\F_p[X]$,
where $\pounds_1(X)$ and $w_p$ are the truncated logarithm and the Wilson quotient.
This is an analogue modulo $p$ of a functional equation, in $\Q[[X]]$, established by Zagier
for the power series $\sum_{n=1}^{\infty}(B_{n}/n)X^n$.
Our proof of the functional equation establishes a connection with a result of Nielsen of 1915,
of which we provide a fresh proof.
Our polynomial framing allows us to derive congruences for certain numerical sums
involving divided Bernoulli numbers.
\end{abstract}

\date{\today}

\maketitle

\section{Introduction}\label{sec:intro}
Let $p$ be a prime. The Artin-Hasse exponential
series is the formal power series in $\Q[[X]]$ defined as
\begin{equation}\label{eq:AH_def}
\AH(X)=\exp\biggl(\sum_{i=0}^{\infty} X^{p^i}/p^{i}\biggr)
=\prod_{i=0}^{\infty}\exp \left(X^{p^i}/p^{i}\right)=\sum_{n=0}^{\infty}u_{n}X^{n}.
\end{equation}

As an immediate application of the Dieudonn\'e-Dwork criterion, its coefficients
are $p$-integral, hence they can be evaluated modulo $p$.
Let $\E_{p}(X)=\sum a_{n}X^{n}$ denote the reduction modulo $p$
of the Artin-Hasse exponential series,
hence thought of as a series in $\F_p[[X]]$.
The coefficients of $\E_p(X)$, and how to compute them, are the main focus of this paper.

The coefficients $u_{n}$ have a remarkable combinatorial interpretation as $u_{n}=h_{n}/n!$, where $h_{n}$ is the number of $p$-power order elements in the symmetric group $S_{n}$. Explicitly
\begin{equation}\label{eq:u_n}
u_{n}=\sum_{n=k_{0}+k_{1}p+\cdots+k_{r}p^r}\frac{1}{\prod_{i=0}^{r}k_{i}!\;p^{i k_{i}}}.
\end{equation}
In particular, $u_{p}=1/p!+1/p$, and hence
$u_{p}=w_p/(p-1)!\equiv -w_p \pmod{p}$, where $w_p$ denotes the Wilson quotient. 

The coefficients $u_n\in\Q$ of the Artin-Hasse series may be computed recursively based on
\begin{equation}\label{eq:FR}
u_{n}=\frac{1}{n}\sum_{i=0}^{\infty}u_{n-p^i},
\end{equation}
where $u_{0}=1$ and we naturally read $u_{m}=0$ for $m<0$.
This recursive formula (see \cite[Lemma~1]{KS}) follows from Equation~\eqref{eq:AH_def} by differentiation.
A natural question is whether and how their residues modulo $p$, that is, the coefficients of $\E_p(X)$,
can be computed from the preceding ones avoiding reference to the rational coefficients of $\AH(X)$.
A version modulo $p$ of Equation~\eqref{eq:FR} may be used as long as $n$ is not a multiple of $p$,
thus reducing the question to the recursive calculation of the coefficients of the form $a_{kp}\in\F_p$.
Note that computing $a_{kp}$ from $u_{kp}$ using~\eqref{eq:u_n} and reducing modulo $p$ requires
supercongruences for binomial coefficients which soon become unavailable.
We give the first few such expressions for $a_{kp}$ in terms of preceding $a_{jp}$
at the beginning of Section~\ref{sec:conjecture}.
A contribution to computing $a_{kp}$ from coefficients $a_{jp}$ with $j<k$
(avoiding explicit reduction of $u_{kp}$ modulo $p$)
is the following result from~\cite{AviMat:G(x)}, but only for $k$ even.

\begin{prop}[Proposition~2 in~\cite{AviMat:G(x)}]\label{prop:X^p}
Let $p$ be an odd prime and $\E_{p}(X)=\sum_{i=0}^{\infty} a_{i}X^{i}$ in $\F_{p}[[X]]$ the reduction modulo $p$ of the Artin-Hasse exponential series.
Then we have
\begin{equation}\label{eq:+-}
\sum_{s=0}^{\infty} a_{sp}X^{sp} \sum_{r=0}^{\infty}(-1)^r a_{rp}X^{rp}\sum_{i=1}^{\infty}X^{p^i}=X^p,
\end{equation}
in $\F_{p}[[X]]$.
\end{prop}

In fact, Proposition~\ref{prop:X^p} determines the value of
$\sum_{r=0}^{k}(-1)^{r}a_{rp}a_{(k-r)p}$
for any $k$, but that sum involves $a_{kp}$ only when $k$ is even.
In particular, for $k<p^2$ that sum vanishes unless $k$ is a multiple of $p-1$,
and equals alternately $1$ and $-1$ otherwise,
as detailed in Proposition~\ref{cor:t_even}.
This paper was born out of an attempt to produce analogous results to cover the case where $k$ is odd.

We present a conjectural equation for the value of
$\sum_{r=0}^{k}(-1)^{r}ra_{rp}a_{(k-r)p}$ in the initial range $k<p$,
which involves {\em divided Bernoulli numbers} $B_n/n$.
We state that as Conjecture~\ref{conj:k_odd}.
It may conveniently be written as
a polynomial congruence involving the power series
$G(X)=\sum_{n=0}^{\infty} (-1)^n a_{np} X^{n}$
and the polynomial in $\F_p[X]$ defined as
\[
\gamma(X)
=\sum_{n=1}^{p-2}(B_{n}/n)X^{p-n}.
\]
A better understanding of this polynomial appears crucial for proving Conjecture~\ref{conj:k_odd}.
The related power series
$\gamma_{0}(X)=\sum_{n=1}^{\infty}(B_{n}/n)X^n\in\Q[[X]]$
was studied by Zagier in~\cite{zagier:bernoulli},
along with other series involving Bernoulli numbers.
Although $\gamma_0(X)$ might not be expressible in closed form, Zagier found a functional equation for it,
which we quote here as Equation~\eqref{eq:gamma_0}.
In Proposition~\ref{prop:feq_gamma} we produce an analogous functional equation
for the polynomial $\gamma(X)$, namely,
\[
\gamma(X-1)-\gamma(X)=\pounds_1(X)+X^{p-1}-w_p-1,
\]
where $\pounds_1(X)$ is the truncated logarithm.

Rather than attempting to adapt Zagier's argument to a proof of
Proposition~\ref{prop:feq_gamma}, we relate that to a congruence of
Nielsen~\cite[Equation~(7) at page~519]{Nielsen}, which is essentially the content of
our Proposition~\ref{prop:Nielsen}.
Because Nielsen's original statement is slightly incorrect,
and in fact was amended in~\cite[Proposition 3.1]{Fouche'}
with a different proof based on $p$-adic integration,
we provide a fresh proof of Nielsen's result
which is close to his original argument.

Our functional equation for $\gamma(X)$,
resulting from this new look at Nielsen's congruence,
allows us to deduce some numerical congruences (modulo $p$) for certain finite sums involving
divided Bernoulli numbers.
We present those in Corollaries~\ref{cor:sixth-roots} and~\ref{cor:eighth-roots}.

\section{On the coefficients of the Artin-Hasse series modulo a prime}

As observed in the Introduction, Equation~\eqref{eq:FR} does not help in calculating
coefficients $a_{rp}$ from previous coefficients.
Our next result allows us to compute some initial coefficients $a_{kp}$ from previous $a_{rp}$
as long as $k$ is even.

\begin{prop}\label{cor:t_even}
Let $p$ be an odd prime and $\E_{p}(X)=\sum_{i=0}^{\infty} a_{i}X^{i}$ in $\F_{p}[[X]]$ the Artin-Hasse exponential series.
Then for $0\leq k <p^2-1$ we have
\begin{equation}\label{eq:cor}
\sum_{r=0}^{k}(-1)^{r}a_{rp}a_{(k-r)p}=\begin{cases} (-1)^j & \textrm{if $k=j(p-1)$,}\\
\\
0 &\textrm{otherwise.}
\end{cases}
\end{equation}
\end{prop}

\begin{proof}
According to Proposition~\ref{prop:X^p} we have
\[
G(X)G(-X)\sum_{i=0}^{\infty}X^{p^{i}-1}=1
\]
in $\F_p[[X]]$,
where $G(X)=\sum_{r=0}^{\infty}(-1)^ra_{rp}X^{r}$.
Because
$\sum_{i=0}^{\infty}X^{p^i-1}\equiv 1+X^{p-1}\pmod{X^{p^2}}$
we deduce
$
G(X)G(-X)
\equiv
1/(1+X^{p-1})
\pmod{X^{p^2}}
$
in $\F_p[[X]]$,
and the conclusion follows after expanding into a geometric series.
\end{proof}

\begin{rem}\label{rem:r_even}
An application of Proposition~\ref{cor:t_even} for $k=2$ yields $a_{2p}=a_{p}^2/2$ when $p>3$,
and $a_{6}=1-a_{3}^2=0$ when $p=3$.
Of course one may reach the same conclusion by expanding $u_{2p}$, the coefficient of $X^{2p}$ in the Artin-Hasse exponential series regarded as series in $\Z_{p}[[X]]$, and then viewing it modulo $p$.
In fact,
\[
u_{2p}=\frac{1+\binom{2p-1}{p-1}(2(p-1)!+(p-1)!^2)}{(2p)!},
\]
and because $(p-1)!=-1+pw_{p}$
the numerator of this expression for $u_{2p}$ equals $1+\binom{2p-1}{p-1}(-1+w_{p}^{2}p^2)$.
When $p>3$ Wolstenholme's Theorem, namely, $\binom{2p-1}{p-1}\equiv 1$ modulo $p^3$, implies $u_{2p}\equiv w_{p}^2/2\equiv u_{p}^2/2\pmod{p}$, as desired.
This argument can be easily reversed to produce
an alternate proof of Wolstenholme's Theorem starting from the identity
$a_{2p}=a_{p}^2/2$
obtained from Equation~\eqref{eq:cor} as above.
\end{rem}

In order to extend Proposition~\ref{cor:t_even} beyond the stated range one may consider, more generally, the sums
\[
s_{k}=\sum_{r=0}^{k}(-1)^r a_{rp}a_{(k-r)p},
\]
for $k\ge 0$, which we conveniently interpret as zero for $k<0$.
Then Equation~\eqref{eq:+-} implies
\[
X^{p}=\sum_{k=0}^{\infty}s_{k}X^{kp}\cdot \sum_{i=1}^{\infty}X^{p^i}=
\sum_{k=0}^{\infty}\left(\sum_{i=0}^{\infty}s_{k+1-p^i}\right)X^{(k+1)p}.
\]
Consequently, we have
\[
\sum_{i=0}^{\infty}s_{k+1-p^i}=\begin{cases}1&\textrm{if $k=0$} \\ 0 &\textrm{otherwise.}\end{cases}
\]
In particular, we deduce that $s_{k}=0$ unless $k$ is a multiple of $p-1$.

\section{A conjecture}\label{sec:conjecture}

Proposition~\ref{cor:t_even} allows one to express the early $a_{kp}$
in terms of preceding coefficients of the form $a_{jp}$ only when $k$ is even.
We already found $a_{2p}=a_p^2/2$ in Remark~\ref{rem:r_even}.
For $a_{3p}$ we need to resort to direct calculation starting from the explicit expression for $u_{3p}$ given by Equation~\eqref{eq:u_n}, which reads
\[
u_{3p}=\frac{1+3\binom{3p-1}{p-1}(p-1)!+3\binom{3p-1}{p-1}\binom{2p-1}{p-1}(p-1)!^2+
\binom{3p-1}{p-1}\binom{2p-1}{p-1}(p-1)!^3}{(3p)!}.
\]
Using known supercongruences for the binomial coefficients involved, shows
\[
a_{3p}=\frac{a_{p}^3}{3!}-\frac{B_{p-3}}{9}
\]
for $p>3$, where $B_{p-3}$ denotes a Bernoulli number.
Now, Proposition~\ref{cor:t_even} yields
\[
a_{4p}=\frac{a_{p}^4}{4!}-\frac{B_{p-3}}{9}a_{p}
\]
for $p>5$.
By similarly applying known supercongruences for binomial coefficients
to the expression for $u_{5p}$ one finds
\[
a_{5p}=\frac{a_{p}^5}{5!}-\frac{B_{p-3}}{9}a_{2p}-\frac{B_{p-5}}{25},
\]
and then Proposition~\ref{cor:t_even} yields
\[
a_{6p}=\frac{a_{p}^6}{6!}-\frac{B_{p-3}}{9}a_{3p}-\frac{B_{p-5}}{25}a_{p}-\frac{B_{p-3}^2}{2\cdot 9^2}
\]
for $p>7$.

To proceed further one would need higher supercongruences for binomial coefficients,
which appear currently unavailable.
In general, obtaining $a_{kp}$ by direct calculation from $u_{kp}$ (for $k$ odd) requires congruences modulo $p^{k+1}$ for the binomial coefficients involved in the numerator of $u_{kp}$.

Alternately, one would need a replacement for Proposition~\ref{cor:t_even} that gives a nontrivial conclusion
for $k$ odd.
The following conjectural recursion formula would do.

\begin{conj}\label{conj:k_odd}
Let $p$ be an odd prime and $\E_{p}(X)=\sum_{i=0}^{\infty} a_{i}X^{i}$ in $\F_{p}[[X]]$ the Artin-Hasse exponential series.
For every integer $1<k<p$ we have
\begin{equation}\label{eq:k_odd}
\sum_{r=0}^{k}(-1)^{r}r a_{rp}a_{(k-r)p}=\frac{B_{p-k}}{k},
\end{equation}
where $B_{n}$ denotes the $n$-th Bernoulli number.
\end{conj}

Here we adopt the prevalent convention where $B_{1}=-1/2$.
For even $k$, Conjecture~\ref{conj:k_odd}
is a consequence of our results.
In fact, in that case, replacing $r$ with $k-r$ in the left-hand of Equation~\eqref{eq:k_odd} and comparing with the original
expression shows
\[
2\sum_{r=0}^{k}(-1)^{r}r a_{rp}a_{(k-r)p}=k \sum_{r=0}^{k}(-1)^{r}a_{rp}a_{(k-r)p}.
\]
However, the right-hand side of this equation vanishes for $0<k<p-1$
according to Proposition~\ref{cor:t_even}.
So does the right-hand side of Equation~\eqref{eq:k_odd}
because of the vanishing of odd-indexed Bernoulli numbers,
with the exception of $B_1$.
The same argument for $k=p-1$ together with Proposition~\ref{cor:t_even}
yields
\[
\sum_{r=0}^{p-1}(-1)^{r}r a_{rp}a_{(p-1-r)p}=-B_{1}.
\]

Thus, Conjecture~\ref{conj:k_odd} remains unproved only for odd $k$.
The value $k=1$ is excluded because the right-hand side is not
$p$-integral.
However, when $k=1$ the left-hand side reads $-a_p$, which is the value modulo $p$ of
$-u_p\equiv w_p\equiv B_{p-1}-(p-1)/p\pmod{p}$,
because of the well-known congruence
\begin{equation}\label{eq:Lehmer}
(1+p B_{p-1})/p\equiv w_p+1\pmod{p},
\end{equation}
see~\cite[Equation~(24)]{lehmer:bernoulli}.
If  Conjecture~\ref{conj:k_odd} holds, then, for example, one would deduce
\[
a_{7p}=\frac{a_{p}^7}{7!}-\frac{B_{p-3}}{9}a_{4p}-\frac{B_{p-5}}{5^2}a_{2p}-\frac{B_{p-3}^2}{2\cdot 9^2}a_{p}-\frac{B_{p-7}}{7^2}
\]
for $p>7$.

Conjecture~\ref{conj:k_odd} can be equivalently formulated as the polynomial congruence
\begin{equation}\label{eq:conj_gamma_A}
XG'(X)G(-X)\equiv w_{p}X-\gamma(X) {\pmod {X^{p}}}
\end{equation}
where $G'(X)$ denotes the derivative of the power series
$G(X)=\sum_{n=0}^{\infty} (-1)^n a_{np} X^{n}$,
and $\gamma(X)$ is the polynomial in $\F_p[X]$ defined as
\[
\gamma(X)
=-\sum_{k=2}^{p-1}(B_{p-k}/k)X^k
=\sum_{n=1}^{p-2}(B_{n}/n)X^{p-n}.
\]

Because $XG(X)G(-X)\equiv X\pmod{X^p}$,
an equivalent form of Equation~\eqref{eq:conj_gamma_A} is
\begin{equation}\label{eq:conj_gamma_B}
X \frac{G'(X)}{G(X)}\equiv w_{p}X-\gamma(X) {\pmod {X^p}}.
\end{equation}
The logarithmic derivative appearing
at the left-hand side suggests looking at
the derivative of $\pounds_1(G(X))$.
Here $\pounds_r(X)=\sum_{k=1}^{p-1}X^k/k^r\in\Q[X]$, for any integer $r$,
and implicitly depending on a given prime $p$, is standard notation
for the {\em finite polylogarithms}.
The {\em truncated logarithm}
$\pounds_1(X)$ is a partial sum of the series $-\log(1-X)$.
Note that we have
$X(d/dX)\pounds_r(X)=\pounds_{r-1}(X)$, and
$\pounds_0(X)=\sum_{k=1}^{p-1}X^k=X(X^{p-1}-1)/(X-1)$.

\begin{lemma}\label{le:pound0_g}
For any $g(X)\in 1+\F_p[[X]]$ we have
\begin{equation*}
\pounds_{0}(g(X))\equiv \begin{cases}
X^{p-1}-1 {\pmod {X^p}} & \textrm{if $g'(0)\neq 0$},\\
-1 {\pmod {X^{2p-2}}}& \textrm{otherwise}.
\end{cases}
\end{equation*}
\end{lemma}

\begin{proof}
Because
$X^p-1=(X-1)(\pounds_0(X)+1)$ we have
\begin{equation*}
(g(X)-1)^{p-1}=\pounds_{0}(g(X))+1.
\end{equation*}
If $g'(0)\neq 0$, then $g(X)-1$ is a multiple of $X$ but not of $X^2$ (in $\F_p[[X]]$).
Since $(g'(0))^{p-1}\in\F_p^\ast$ the power series
$(g(X)-1)^{p-1}$
belongs to
$X^{p-1}+X^p\F_p[[X]]$,
and hence
$\pounds_{0}(g(X))\equiv X^{p-1}-1 {\pmod {X^p}}$
as desired.
If, however, $g'(0)=0$, then $g(X)-1$ is a multiple of $X^2$, and hence
$\pounds_{0}(g(X))\equiv -1 {\pmod {X^{2p-2}}}$.
\end{proof}

In particular, taking $g(X)=G(X)=\sum_{n=0}^{\infty}(-1)^n a_{np}X^n$
in Lemma~\ref{le:pound0_g}
we deduce
$X\pounds_{0}(G(X))\equiv -X\pmod{X^p}$.
Hence
\[
X\frac{d}{dX}\pounds_1(G(X))
=X\frac{\pounds_{0}(G(X))}{G(X)}G'(X)
\equiv
-X\frac{G'(X)}{G(X)}\pmod{X^p},
\]
and so our conjectured Equation~\eqref{eq:conj_gamma_A} amounts to
\begin{equation}\label{eq:conj_G_C}
X\frac{d}{dX}\pounds_1(G(X))\equiv -w_{p}X+\gamma(X) {\pmod {X^p}}.
\end{equation}
This suggests integrating both sides after dividing them by $X$.
Set
\[
\rho(X)=\sum_{n=1}^{p-2}\frac{B_{n}}{n^2}X^{p-n},
\]
whence $X(d/dX) \rho(X)=-\gamma(X)$.
Noting that the constant term of $\pounds_1(G(X))$ is $\pounds_1(1)=0$,
yet another equivalent form of our conjectured Equation~\eqref{eq:conj_gamma_A} reads
\begin{equation}\label{eq:conj_G_D}
\pounds_1(G(X))\equiv -w_{p}X-\rho(X) {\pmod {X^{p}}}.
\end{equation}
Unfortunately, the polynomial $\rho(X)$ remains rather mysterious.
However, in the next section we will obtain some information on the polynomial $\gamma(X)$.

\section{A polynomial involving divided Bernoulli numbers}\label{sec:Bernoulli}

Several polynomial forms of our Conjecture~\ref{conj:k_odd}
given in the last section involve the polynomial
$
\gamma(X)
=\sum_{n=1}^{p-2}(B_{n}/n)X^{p-n}
\in\F_p[X]$.
Don Zagier noted in~\cite{zagier:bernoulli} that
the related power series
$\gamma_{0}(X)=\sum_{n=1}^{\infty}(B_{n}/n)X^n\in\Q[[X]]$
satisfies the functional equation
\begin{equation}\label{eq:gamma_0}
\gamma_{0}\left(\frac{X}{1-X}\right)-\gamma_{0}(X)=\log (1-X)+X.
\end{equation}
Because
$\bigl(X/(1-X)\bigr)^n=\sum_{k=1}^{\infty}\binom{k-1}{n-1}X^{k}$,
Equation~\eqref{eq:gamma_0} is equivalent to the identity
\begin{equation}\label{eq:sum_bernoulli}
\sum_{n=1}^{k-1}\binom{k-1}{n-1}\frac{B_{n}}{n}=-\frac{1}{k},
\quad\text{for $k>1$.}
\end{equation}
Note that $\bigl(B_{p-1}/(p-1)\bigr)X^{p-1}$ is the earliest term of the series $\gamma_{0}(X)$
whose coefficient is not in $\Z_p$, but this term cancels in the difference at the left-hand side of
Equation~\eqref{eq:gamma_0}.
Thus, after viewing Equation~\eqref{eq:gamma_0} modulo $X^p$ we may view the result modulo $p$,
and find
\begin{equation}\label{eq:feq_gamma0}
\gamma^\ast\left(\frac{X}{1-X} \right)-\gamma^\ast(X)\equiv X-\pounds_1(X) \pmod {X^p},
\end{equation}
in terms of the polynomial
$\gamma^\ast(X)=\sum_{n=1}^{p-2}(B_{n}/n)X^{n}=X^p\gamma(1/X)$ in $\F_{p}[X]$.

One can actually refine the congruence of Equation~\eqref{eq:feq_gamma0}
to an identity in $\F_p[[X]]$
after an appropriate modification, and even to a polynomial identity
when expressed in terms of $\gamma(X)$, namely, Equation~\eqref{eq:feq_gamma} below.
Although it is possible to prove Equation~\eqref{eq:feq_gamma} by adapting direct arguments
in~\cite{zagier:bernoulli} that rely on Equation~\eqref{eq:sum_bernoulli},
we take a different route.

A polynomial closely related to $\gamma(X)$, namely,
$\sum_{k=1}^{p-2}(B_{k}/k)(X^{p-1-k}-1)$,
was considered by Nielsen in~\cite{Nielsen},
and in more recent times by Fouch\'e starting with~\cite{Fouche'}.
We quote a crucial congruence for Nielsen's work,
in an equivalent formulation that is closer to our polynomial $\gamma(X)$.

\begin{prop}\label{prop:Nielsen}
For every integer $x$ we have
\[
\sum_{k=1}^{p-2}(B_{k}/k)x^{p-k}
\equiv x\,q_{p}(x)+w_{p}\,x
+\lfloor x/p\rfloor
\pmod{p}.
\]
\end{prop}

Here
$q_{p}(x)=(x^{p-1}-1)/p$, a Fermat quotient.
The product $x\,q_{p}(x)$ is sometimes easier to handle than $q_p(x)$,
as it is an integer for every integer $x$.
Proposition~\ref{prop:Nielsen} is
essentially~\cite[Proposition 3.1]{Fouche'}, an amended version of Nielsen's congruence.
As Fouch\'e pointed out,
the congruence as stated by Nielsen,
namely, \cite[Equation~(7) at page~519]{Nielsen},
also quoted in~\cite[page~112]{Dickson1},
is only correct in the range $0<x<p$.
Fouch\'e provided a proof based on $p$-adic integration.
For the reader's convenience we present a short proof
of Proposition~\ref{prop:Nielsen}
that is close to Nielsen's original argument.

\begin{proof}
First note that the statement follows inductively from its special case
where $0\le x<p$, because
$(x+p)\,q_p(x+p)
\equiv
x\,q_p(x)-1
\pmod{p}
$.
According to Faulhaber's formula, for any non-negative integer $x$ we have
\begin{align*}
\sum_{k=0}^{x}k^{p-1}
&=
\frac{1}{p}\sum_{j=0}^{p-1}(-1)^j\binom{p}{j}B_j x^{p-j}
\\&\equiv
\frac{x^p+pB_{p-1}x}{p}
-\sum_{j=1}^{p-2}(B_j/j) x^{p-j}
\pmod{p}.
\\&\equiv
x\,q_p(x)+(w_p+1)x
-\sum_{j=1}^{p-2}(B_j/j) x^{p-j}
\pmod{p},
\end{align*}
where we have used
$\frac{1}{p}\binom{p}{j}=\frac{1}{j}\binom{p-1}{j-1}\equiv(-1)^{j-1}/j\pmod{p}$
for $0<j<p$, and then Equation~\eqref{eq:Lehmer}.
The desired conclusion follows because
$\sum_{k=0}^{x}k^{p-1}\equiv x\pmod{p}$
for $0\le x<p$.
\end{proof}

The left-hand side of Nielsen's congruence in
Proposition~\ref{prop:Nielsen}
resembles an evaluation of our polynomial $\gamma(X)$, but the expression
$x\,q_p(x)$ at the right-hand side
cannot be viewed as a polynomial with coefficients in $\F_p$,
as it depends on the value of $x$ modulo $p^2$.
Although the expression
$x\,q_p(x)+\lfloor x/p\rfloor$
can, handling Fermat quotients as polynomials in $\F_p[X]$ is more elegantly done by considering the symmetrized expression
\begin{equation}\label{eq:Granville}
x\,q_{p}(x)+(1-x)\,q_{p}(1-x)=\frac{x^p-1+(1-x)^p}{p}\equiv -\pounds_1(x)\pmod{p},
\end{equation}
see~\cite[Equation~12]{MatTau:polylog}, for example.
Now
$\lfloor x/p\rfloor+\lfloor (1-x)/p\rfloor$
equals $0$ if $x\equiv 0,1\pmod{p}$, and $-1$ otherwise,
and the polynomial
$1-X^{p-1}-(1-X)^{p-1}\in\F_p[X]$ takes the value $0$ on $0,1$,
and $-1$ on the remaining elements of $\F_p$.
It follows that $\gamma(X)$ satisfies
\begin{equation}\label{eq:feq_gamma_sym}
\gamma(X)+\gamma(1-X)=-\pounds_1(X)-X^{p-1}-(1-X)^{p-1}+w_p+1,
\end{equation}
because both sides have degree less than $p$, and agree when
evaluated on elements of $\F_p$ according to Proposition~\ref{prop:Nielsen}.
Using the fact that $\gamma(X)+X^{p-1}/2$ is an odd polynomial turns
Equation~\eqref{eq:feq_gamma_sym}
into the following finite analogue of Equation~\eqref{eq:gamma_0}
found by Zagier.

\begin{prop}\label{prop:feq_gamma}
The polynomial
$\gamma(X)
=\sum_{n=1}^{p-2}(B_{n}/n)X^{p-n}$
in $\F_p[X]$
satisfies the functional equation
\begin{equation}\label{eq:feq_gamma}
 \gamma(X-1)-\gamma(X)=\pounds_1(X)+X^{p-1}-w_p-1.
\end{equation}
\end{prop}

Equation~\eqref{eq:feq_gamma} actually characterizes $\gamma(X)$ uniquely among the
polynomials in $\F_p[X]$ of degree less than $p$ and without constant term.
That is because one easily sees that a polynomial
$f(X)\in\F_p[X]$ satisfies $f(X-1)-f(X)=0$
if and only if it has the form
$f(X)=g(X^p-X)$ for some $g(X)\in\F_p[X]$.

It is also easy to deduce Proposition~\ref{prop:Nielsen} from
Proposition~\ref{prop:feq_gamma},
working inductively
from the case $x=1$, which follows from
Equation~\eqref{eq:feq_gamma} because
$\gamma(0)=0$ and $\pounds_1(1)=0$.

Despite this logical equivalence, Proposition~\ref{prop:feq_gamma}
(or Equation~\eqref{eq:feq_gamma_sym})
is more flexible than Proposition~\ref{prop:Nielsen},
because of its formulation in terms of polynomials
allows evaluation on $p$-integral algebraic numbers.

We present some applications where $\pounds_1(X)$
admits nice evaluations modulo $p$ on those numbers,
such as
$
\pounds_1(1)\equiv 0\pmod{p},
$
$
\pounds_1(1/2)\equiv q_p(2)\pmod{p},
$
and
$
\pounds_1(-\omega)\equiv 0\pmod{p},
$
where $\omega=(-1+i\sqrt{3})/2$
(for $p>2$ and $p>3$ in the second and third case).
All these follow easily from
\begin{equation}\label{eq:Granville_pol}
\pounds_1(X)\equiv\frac{1-X^p-(1-X)^p}{p}\pmod{p},
\end{equation}
a polynomial formulation of Equation~\eqref{eq:Granville}.
Note also the symmetries
$
\pounds_1(1-X)\equiv\pounds_1(X)\pmod{p},
$
which follows from Equation~\eqref{eq:Granville_pol}, and
$
\pounds_1(1/X)\equiv -\pounds_1(X)/X^p\pmod{p}
$.

Thus, taking $X=1$ in Equation~\eqref{eq:feq_gamma_sym}
or~\eqref{eq:feq_gamma} yields
\[
\sum_{n=1}^{p-2}B_n/n
\equiv w_p\pmod{p},
\]
and taking $X=1/2$ in Equation~\eqref{eq:feq_gamma_sym} yields
\[
\sum_{n=1}^{p-2}2^n B_n/n
\equiv -q_p(2)+w_p-1\pmod{p}.
\]

Before dealing with the slightly more complicated application
where $X=-\omega$,
recall that $\gamma(X)$ may be more conveniently written as
$\gamma(X)=-\sum_{n=2}^{p-1}(B_{p-n}/n)X^n$, and that the term for
$n=2$ vanishes if $p>3$.

\begin{cor}\label{cor:sixth-roots}
For every prime $p>3$ we have
\[
\sum_{m=0}^{\lfloor (p-5)/6\rfloor}
\frac{B_{p-6m-3}}{2m+1}
\equiv
\frac{1}{4}-\frac{3}{4}\Bigl(\frac{p}{3}\Bigr)
\pmod{p}.
\]
Here $(\frac{p}{3})$ denotes a Legendre symbol, hence the sum is congruent to $1$
or $-1/2$ modulo $p$ according as $p\equiv \pm 1\pmod{3}$.
\end{cor}

\begin{proof}
Let $\omega=(-1+i\sqrt{3})/2$, hence $-\omega$ and $-\omega^{-1}$ are the primitive
complex sixth roots of unity.
We will evaluate Equation~\eqref{eq:feq_gamma}
on $X=-\omega^{-1}=\omega+1$ and its reciprocal (which equals its complex conjugate).
Because
$\sum_{j=0}^5(-1)^j(-\omega)^{jn}$ equals $6$ if $n\equiv 3\pmod{6}$
and zero otherwise,
we have
\begin{align*}
\gamma(1)
&
-\gamma(-\omega)
+\gamma(\omega^{-1})
-\gamma(-1)+\gamma(\omega)-\gamma(-\omega^{-1})
\\&=
-6\sum_{\substack{3\le n\le p-1,\\ n\equiv 3\pmod{6}}}
\frac{B_{p-n}}{n}
=
-2\sum_{m=0}^{\lfloor (p-5)/6\rfloor}
\frac{B_{p-6m-3}}{2m+1}
\end{align*}

Because
$\pounds_1(-\omega)\equiv\pounds_1(-\omega^{-1})\equiv 0\pmod{p}$
as noted earlier,
according to Equation~\eqref{eq:feq_gamma} we have
$
\gamma(\omega)-\gamma(-\omega^{-1})
=
(-\omega^{-1})^{p-1}-w_p-1
$,
and
$
\gamma(\omega^{-1})-\gamma(-\omega)
=
(-\omega)^{p-1}-w_p-1
$.
Equation~\eqref{eq:feq_gamma} also implies
$\gamma(1)=w_p$ and $\gamma(-1)=-w_p-1$.
In conclusion, we find
\begin{align*}
\gamma(1)
-\gamma(-\omega)
+\gamma(\omega^{-1})
-\gamma(-1)+\gamma(\omega)-\gamma(-\omega^{-1})
&=
\omega^{p-1}
+\omega^{1-p}
-1
\\&=
-\frac{1}{2}+\frac{3}{2}\Bigl(\frac{p}{3}\Bigr),
\end{align*}
and the desired congruence follows.
\end{proof}

For our final application of the functional equation(s) for $\gamma(X)$
we need the congruence
$\pounds_1(i)+\pounds_1(-i)\equiv -q_p(2)\pmod{p}$,
which can be proved as follows starting from Equation~\eqref{eq:Granville_pol}, with congruences taking place in the ring of algebraic integers:
\begin{align*}
\pounds_1(i)+\pounds_1(-i)
&\equiv
\frac{2-(1-i)^p-(1+i)^p}{p}
\pmod{p},
\\&=
\frac{2-2\cdot 2^{(p-1)/2}\cdot (-1)^{(p^2-1)/8}}{p}
\\&=
-q_p(2)
+
\frac{\bigl(2^{(p-1)/2}-(-1)^{(p^2-1)/8}\bigr)^2}{p}
\\&\equiv
-q_p(2)\pmod{p},
\end{align*}
where the last congruence follows from
$2^{(p-1)/2}\equiv (-1)^{(p^2-1)/8}\pmod{p}$,
due to the quadratic character of $2$.

\begin{cor}\label{cor:eighth-roots}
For every prime $p>3$ we have
\[
-\sum_{n=3}^{p-2}\frac{B_{p-n}}{n}\cdot (-1)^{(n^2-1)/8}\cdot 2^{(n+1)/2}
\equiv
q_{p}(2)+2 w_p+
3\frac{1-(-1)^{(p-1)/2}}{2}
\pmod{p}.
\]
\end{cor}

\begin{proof}
Similarly to the proof of Corollary~\ref{cor:sixth-roots},
we will evaluate Equation~\eqref{eq:feq_gamma_sym}
on $X=i$ and $X=-i$, and add together the resulting equations.

We first deal with the quantities at the left-hand side of  Equation~\eqref{eq:feq_gamma_sym}.
Because $\gamma(X)+\gamma(-X)=-X^{p-1}$ we have
$\gamma(i)+\gamma(-i)=(-1)^{(p+1)/2}$.
Next, because
\[
(1-i)^n+(1+i)^n
=2^{-n/2}\left(\biggl(\frac{1-i}{\sqrt{2}}\biggr)^n+\biggl(\frac{1-i}{\sqrt{2}}\biggr)^n\right)
=(-1)^{(n^2-1)/8}\cdot 2^{(n+1)/2}
\]
for any odd integer $n$, we find
\begin{align*}
\gamma(1-i)
&+\gamma(1+i)
+\frac{1}{2}(1-i)^{p-1}+\frac{1}{2}(1+i)^{p-1}
\\&=
-\sum_{n=3}^{p-2}\frac{B_{p-n}}{n}(1-i)^n
-\sum_{n=3}^{p-2}\frac{B_{p-n}}{n}(1+i)^n
\\&=
-\sum_{n=3}^{p-2}\frac{B_{p-n}}{n}\cdot (-1)^{(n^2-1)/8}\cdot 2^{(n+1)/2}.
\end{align*}

Turning to the right-hand side of Equation~\eqref{eq:feq_gamma_sym},
we have
$\pounds_1(i)+\pounds_1(-i)\equiv -q_p(2)\pmod{p}$
as noted earlier, and
$(\pm i)^{p-1}=(-1)^{(p-1)/2}$.
Furthermore, because
$(1\pm i)^p\equiv 1\pm i^p=1\pm(-1)^{(p-1)/2}i\pmod{p}$
we find
\[
(1-i)^{p-1}+(1+i)^{p-1}\equiv
1+(-1)^{(p-1)/2}\pmod{p}.
\]
Therefore,
evaluating the right-hand side of Equation~\eqref{eq:feq_gamma_sym}
on $X=i$ and $X=-i$ and adding the results yields
\[
q_p(2)
+2w_p+1
-3(-1)^{(p-1)/2}.
\]
Equating this to the sum of the left-hand side of Equation~\eqref{eq:feq_gamma_sym}
evaluated on $X=i$ and $X=-i$, as discussed earlier,
yields the desired conclusion.
\end{proof}

\bibliography{References}

\def\cprime{$'$} \def\polhk#1{\setbox0=\hbox{#1}{\ooalign{\hidewidth
  \lower1.5ex\hbox{`}\hidewidth\crcr\unhbox0}}}
\providecommand{\bysame}{\leavevmode\hbox to3em{\hrulefill}\thinspace}
\providecommand{\MR}{\relax\ifhmode\unskip\space\fi MR }
\providecommand{\MRhref}[2]{%
  \href{http://www.ams.org/mathscinet-getitem?mr=#1}{#2}
}
\providecommand{\href}[2]{#2}
\begin{thebibliography}{Fou85}

\bibitem[AM]{AviMat:G(x)}
M.~Avitabile and S.~Mattarei, \emph{The {A}rtin-{H}asse series and {L}aguerre
  polynomials modulo a prime}, preprint, {\sf arXiv:2308.14736}.

\bibitem[Dic66]{Dickson1}
Leonard~Eugene Dickson, \emph{History of the theory of numbers. {V}ol. {I}:
  {D}ivisibility and primality.}, Chelsea Publishing Co., New York, 1966.
  \MR{0245499 (39 \#6807a)}

\bibitem[Fou85]{Fouche'}
Willem Fouch\'{e}, \emph{A reciprocity law for polynomials with {B}ernoulli
  coefficients}, Trans. Amer. Math. Soc. \textbf{288} (1985), no.~1, 59--67.
  \MR{773047}

\bibitem[KS79]{KS}
Kiyomi Kanesaka and Koji Sekiguchi, \emph{Representation of {W}itt vectors by
  formal power series and its applications}, Tokyo J. Math. \textbf{2} (1979),
  no.~2, 349--370. \MR{560274}

\bibitem[Leh38]{lehmer:bernoulli}
Emma Lehmer, \emph{On congruences involving {B}ernoulli numbers and the
  quotients of {F}ermat and {W}ilson}, Ann. of Math. (2) \textbf{39} (1938),
  no.~2, 350--360. \MR{1503412}

\bibitem[MT13]{MatTau:polylog}
Sandro Mattarei and Roberto Tauraso, \emph{Congruences for central binomial
  sums and finite polylogarithms}, J. Number Theory \textbf{133} (2013), no.~1,
  131--157. \MR{2981405}

\bibitem[Nie15]{Nielsen}
N.~Nielsen, \emph{Sur les nombres de \emph{Bernoulli} et leur application dans
  la th{\'e}orie des nombres.}, Overs. {Danske} {Vidensk}. {Selsk}. {Forh}.
  1915, 509-524 (1915)., 1915.

\bibitem[Zag14]{zagier:bernoulli}
Don Zagier, \emph{Appendix: {C}urious and {E}xotic {I}dentites for {B}ernoulli
  {N}umbers}, Bernoulli numbers and zeta functions, Springer Monographs in
  Mathematics, Springer, Tokyo, 2014, pp.~xii+274. \MR{3307736}

\end{thebibliography}

\end{document}